%% file: main.tex
\documentclass[11pt,twoside]{article}

\usepackage[a4paper,margin=3.5cm]{geometry}
\usepackage[utf8]{inputenc}
\usepackage[english]{babel}
\usepackage{csquotes}
\usepackage{amssymb}
\usepackage{amsmath}
\usepackage{amsthm}
\usepackage{dsfont}
\usepackage{enumitem}
\usepackage{epigraph}
\usepackage{fancyhdr}
\usepackage{mathrsfs}
\usepackage{mathtools}
\usepackage{graphicx}
\usepackage{array}
\usepackage{pgf,tikz}
\usepackage{tcolorbox}
\usepackage{tikz-cd}
\usepackage{titling}
\usepackage{url}
\usepackage{verbatim}
\usepackage[style=numeric-comp, backend=biber]{biblatex}
\usepackage{hyperref}
\hypersetup{
    colorlinks=true,
    allcolors=black
}

\setlist{noitemsep}

\usetikzlibrary{arrows}
\usetikzlibrary{decorations.pathmorphing}
\usetikzlibrary{patterns}

\newcommand{\F}{\mathcal{F}}

\newcommand{\G}{\mathcal{G}}

\newcommand{\mc}{\mathcal}

\newcommand{\OO}{\mathcal{O}}
\newcommand{\op}{\tn{op}}

\renewcommand{\P}{\mathbb{P}}

\newcommand{\ra}{\rightarrow}

\newcommand{\tn}{\text}

\newcommand{\Z}{\mathbb{Z}}

\DeclareMathOperator{\can}{can}

\DeclareMathOperator{\ev}{ev}

\DeclareMathOperator{\GGW}{\mathbb{G}W}
\DeclareMathOperator{\GW}{GW}

\DeclareMathOperator{\hofib}{hofib}

\DeclareMathOperator{\id}{id}

\DeclareMathOperator{\K}{K}

\DeclareMathOperator{\Perf}{Perf}

\DeclareMathOperator{\pr}{pr}

\DeclareMathOperator{\sheafhom}{\mathscr{H}\textit{\kern -4pt om}\,}

\DeclareMathOperator{\Spec}{Spec}
\DeclareMathOperator{\sPerf}{sPerf}

\DeclareMathOperator{\Vect}{Vect}
\DeclareMathOperator{\W}{W}

\numberwithin{equation}{subsection}

\theoremstyle{definition}
\newtheorem{definition}{Definition}[subsection]

\newtheorem{remark}[definition]{Remark}

\theoremstyle{plain}
\newtheorem{corollary}[definition]{Corollary}
\newtheorem{lemma}[definition]{Lemma}
\newtheorem{proposition}[definition]{Proposition}
\newtheorem{theorem}[definition]{Theorem}
\newtheorem{maintheorem}{Theorem}

\title{\textsc{the projective bundle formula for grothendieck-witt spectra}}
\author{
    \textsc{herman rohrbach}
    \thanks{
        The author is supported by the research training group \emph{GRK 2240: Algebro-Geometric Methods in Algebra, Arithmetic and Topology}.
    }
}
\predate{}
\postdate{}
\date{}

\begin{document}

\maketitle

\input{Sections/Abstract}

\input{Sections/Introduction.tex}

\input{Sections/ProjectiveBundleFormula.tex}

\printbibliography

\end{document}

%% file: Sections/Abstract.tex
\begin{abstract}
    Grothendieck-Witt spectra represent higher Grothendieck-Witt groups and higher Hermitian K-theory in particular.
    A description of the Grothendieck-Witt spectrum of a finite dimensional projective bundle $\P(\mc{E})$ over a base scheme $X$ is given in terms of the Grothendieck-Witt spectra of the base, using the dg category of strictly perfect complexes, provided that $X$ is a scheme over $\Spec \Z[1/2]$ and satisfies the resolution property, e.g. if $X$ has an ample family of line bundles.
\end{abstract}

%% file: Sections/Introduction.tex
\section{Introduction}
    \label{section:introduction}

\emph{Projective bundles} are of central importance in algebraic geometry and its applications, and any computation of an invariant of projective bundles is bound to be useful for further computations of that same invariant, be it by using the result for projective bundles directly or by employing techniques similar to the ones found in the proof of the projective bundle formula.

The invariant at hand is the \emph{Grothendieck-Witt spectrum}, the stable homotopy groups of which are called \emph{higher Grothendieck-Witt groups}.
A related invariant is the \emph{Witt spectrum}, the stable homotopy groups of which coincide with the negative Grothendieck-Witt groups.
In algebraic geometry, algebraic K-theory is defined by considering vector bundles on a scheme $X$ and constructing an associated K-theory spectrum $\K(X)$. 
Similarly, Grothendieck-Witt theory, also known as Hermitian K-theory, is defined using vector bundles on $X$ which are additionally equipped with a \emph{symmetric form}.
Grothendieck-Witt theory subsumes algebraic K-theory, but is also harder to compute.

When the base scheme $X$ is a point $\Spec k$ for some field $k$, then $\K(X)$ is the algebraic K-theory of $k$-vector spaces, while the Grothendieck-Witt spectrum $\GW(X)$ of $X$ is the Grothendieck-Witt theory of $k$-vector spaces equipped with a quadratic form. 
The latter theory was founded by Ernst Witt in the 1930s, and both theories remain topics of interest to this day.

As for projective bundle formulas, it is known that orientable cohomology theories $E$, on the one hand, satisfy
\begin{equation*}
    E(\P(\mc{E})) \simeq \bigoplus_{i=0}^r E(X) \alpha^i,
\end{equation*}
where $\P(\mc{E})$ is a projective bundle of dimension $r$ over a base scheme $X$ and $\alpha$ is a class in $E(\P(\mc{E}))$ satisfying some polynomial relation which determines a multiplicative structure.
Algebraic K-theory is an example of an orientable cohomology theory.

Grothendieck-Witt theory, on the other hand, is among the first examples of non-orientable cohomology theories.
A first relevant result in \cite{arason80} states that the classical Witt group of a projective space of dimension $r$ over a field $k$ satisfies
\begin{equation*}
    \W(\P^r_k) \cong \W(k),
\end{equation*}
which shows that a projective bundle formula for Grothendieck-Witt theory must be different than that of orientable theories.

If $X$ is a scheme over $\Spec \Z[1/2]$, then one says that $2$ is invertible in $X$.
Note that if $X = \Spec k$, this coincides with the usual notion of $2$ being invertible in $k$, in which case quadratic forms correspond bijectively to symmetric bilinear forms.
There is an analogue of this correspondence for schemes $X$ in which $2$ is invertible.
This simplifies Grothendieck-Witt theory significantly, and therefore the current paper is restricted to this case.

In \cite{walter03projective}, projective bundle formulas are proved for the classical Grothendieck-Witt group, and it is claimed in \cite[remark 9.11]{schlichting17} that these results immediately generalize to Grothendieck-Witt spectra, although a rigorous proof is only provided for the case of $\P^1_X$.
Even if the main results of this paper are not new, the techniques used to prove them are, and could be useful for future computations.
A detailed proof of a projective bundle formula for Grothendieck-Witt spectra is provided for the general case of a projective bundle over a base scheme $X$, thus filling a hiatus in the literature.
The proof contained in this paper uses a combination of the techniques of \cite{walter03projective}, \cite{schlichting17} and \cite{balmer05koszul}.
The present results, in the special case of trivial projective bundles, have recently been used in the proof of \cite[theorem 5.1]{karoubi20}, which contains a computation of the Grothendieck-Witt spectrum of a punctured affine space over $X$.

The projective bundle formulae for Grothendieck-Witt spectra are formulated as two main theorems.
Theorem \ref{maintheorem:projectivebundleformulaforgwwithtwistofdifferentparity} is harder to prove than theorem \ref{maintheorem:projectivebundleformulawithtwistofsameparityasdimension}, due to the parities of $m$ and $r$. 
In the statements of these theorems, $X$ is a scheme over $\Spec\Z[1/2]$ satisfying the resolution property, $\mc{E}$ is a locally free $\OO_X$-module of rank $r+1$, $s = \lceil r/2 \rceil$, $\P = \P(\mc{E})$ is the projective bundle with projection map $\pi: \P \rightarrow X$, $\mc{L}$ is an invertible $\OO_X$-module, $m,n \in \Z$ and $\pi^*\mc{L}(m) = \OO_{\P}(m) \otimes \pi^*\mc{L}$ is the $m$-th twist of $\pi^*\mc{L}$. 

\begin{maintheorem} \label{maintheorem:projectivebundleformulawithtwistofsameparityasdimension}
The following statements hold.
\begin{enumerate}[label=(\roman*)]
    \item If $m$ and $r$ are even, then there is a stable equivalence of spectra
    \begin{equation*}
        \begin{tikzcd}[cramped]
            \GW^{[n]}(X, \mc{L}) \oplus \K(X)^{\oplus s} \arrow[r]
                & \GW^{[n]}(\P, \pi^*\mc{L}(m)).
        \end{tikzcd}    
    \end{equation*}
    \item If $m$ and $r$ are odd, then there is a stable equivalence of spectra
    \begin{equation*}
        \begin{tikzcd}[cramped]
            \K(X)^{\oplus s} \arrow[r]
                & \GW^{[n]}(\P, \pi^*\mc{L}(m)).
        \end{tikzcd}    
    \end{equation*}
\end{enumerate}
\end{maintheorem}

\begin{maintheorem} \label{maintheorem:projectivebundleformulaforgwwithtwistofdifferentparity}
The following statements hold.
\begin{enumerate}[label=(\roman*)]
    \item If $r$ is even and $m$ is odd, then there is a split homotopy fibration
    \begin{equation*}
        \begin{tikzcd}[column sep=small]
            K(X)^{\oplus s-1} \arrow[r]
                & \GW^{[n]}(\P, \pi^*\mc{L}(m)) \arrow[r]
                    & \GW^{[n-r]}(X,\det \mc{E}^{\vee} \otimes \mc{L}).
        \end{tikzcd}
    \end{equation*}
    \item If $r$ is odd and $m$ is even, then there is a homotopy fibration
    \begin{equation*}
        \begin{tikzcd}[column sep=tiny]
            \GW^{[n]}(X, \mc{L}) \oplus K(X)^{\oplus s-1} \arrow[r]
                &[-0.1em] \GW^{[n]}(\P, \pi^*\mc{L}(m)) \arrow[r]
                    &[-0.1em] \GW^{[n-r]}(X,\det \mc{E}^{\vee} \otimes \mc{L}),
        \end{tikzcd}
    \end{equation*}
    which splits if the element $\nu'$ of lemma \ref{lemma:middletermexterioralgebraissymmetric} vanishes in $\W^{[r+1]}_0(X, \det \mc{E})$.
    The condition for the splitting is satisfied e.g. if $\mc{E}$ is a trivial bundle.
\end{enumerate}
\end{maintheorem}

This paper is organized as follows.
Section \ref{subsection:dualityonthecategoryofperfectcomplexes} contains background on the dg category $\Perf(X)$ of perfect complexes on a scheme $X$, and the dualities on this category.
In section \ref{subsection:constructingsymmetricformsfromkoszulcomplexes}, the necessary technical tools for the proofs of the main theorems are developed; these consist mainly of certain symmetric forms that are constructed from the Koszul complex.
The most involved construction is that of theorem \ref{theorem:middletermofkoszulcomplexwittnilcutskoszulintwo} and provides the splitting of the homotopy fibration found in theorem \ref{maintheorem:projectivebundleformulaforgwwithtwistofdifferentparity}(ii).
Section \ref{subsection:localizationandadditivityforgrothendieckwittspectra} recalls some important results from \cite{schlichting17} and casts them into a form that is better suited to this paper.
Finally, main theorems \ref{maintheorem:projectivebundleformulawithtwistofsameparityasdimension} and \ref{maintheorem:projectivebundleformulaforgwwithtwistofdifferentparity} are proved in section \ref{subsection:theprojectivebundleformulaforgrothendieckwittspectra} as theorems \ref{theorem:projectivebundleformulawithtwistofsameparityasdimension} and \ref{theorem:projectivebundleformulaforgwwithtwistofdifferentparity}, respectively.

\subsection*{Acknowledgements}
    \label{subsection:acknowledgements}
    
I am deeply indebted to Marco Schlichting and Heng Xie for providing key ideas behind the present results, as well as for helpful discussions to clarify these ideas.
One idea that stands out in particular is that to consider symmetric forms on the Koszul complex.
Additionally, I would like to thank Jens Hornbostel, Marcus Zibrowius and Sean Tilson for many helpful discussions.

%% file: Sections/ProjectiveBundleFormula.tex
\section{The projective bundle formula for Grothendieck-Witt spectra}
    \label{section:theprojectivebundleformulaforgrothendieckwittspectra}
    
\subsection{Duality on the dg category of perfect complexes}
    \label{subsection:dualityonthecategoryofperfectcomplexes}
    
This section introduces the general yoga of duality on the dg category of perfect complexes on a scheme.
In the next section, this yoga is then used to construct some specific objects needed in the proof of the main theorems.
    
The sign conventions used will be the same as those in \cite[section 1.11]{schlichting17}.
Let $X$ be a scheme with the resolution property, that is, satisfying $\Perf(X) \simeq \sPerf(X)$.
An important class of schemes that satisfy this property is that of schemes with an ample family of line bundles.
For ease of notation, let $\OO = \OO_X$.
All tensor products in this section are taken over $\OO$, unless indicated otherwise.
Let $\Perf(X)$ be the usual closed symmetric monoidal pretriangulated dg category of perfect complexes of $\OO$-modules, where the monoidal structure is given by the tensor product of complexes and the monoidal unit is the perfect complex $\OO$, concentrated in cohomological degree $0$. 

It is useful to think of $\Perf(X)$ as a subcategory of the category of dg $\OO$-modules.
As $\Perf(k)$ is the subcategory of the dg category of chain complexes of $k$-vector spaces consisting of the perfect complexes, so $\Perf(X)$ is the subcategory of the dg category of chain complexes of vector bundles on $X$ consisting of the perfect complexes; the sheaf of rings $\OO$ takes the place of the commutative ring $k$.

The following conventions are followed here:
\begin{enumerate} [label=(\roman*)]
    \item complexes of $\OO$-modules are cohomologically graded;
    \item for an object $A$ of $\Perf(X)$ with differential $d$, the $i$-th differential is $d_i:A_i \ra A_{i+1}$;
    \item for $n \in \Z$, the $n$-th shift of $A$ is denoted by $A[n]$, with $A[n]_i = A_{i+n}$;
    \item a homogeneous morphism $f: A \ra B$ in $\Perf(X)$ has degree $j$ if its components are $f_i:A_i \ra B_{i+j}$ for all $i \in \Z$; and
    \item since $\Perf(X) \simeq \sPerf(X)$ by assumption, an object $A$ of $\Perf(X)$ will always be assumed to be a cohomologically bounded complex of locally free $\OO$-modules.
\end{enumerate}

\begin{definition} \label{def:dualityonperf}
Let $A$ be an object of $\Perf(X)$.
The \emph{duality defined by $A$} consists of the following functor $\vee_A: \Perf(X)^{\op} \rightarrow \Perf(X)$ and natural transformation $\can^A: \id \rightarrow \vee_A \circ \vee_A^{\op}$:
\begin{enumerate}[label=(\roman*)]
    \item on objects, $\vee_A$ is given by $F \mapsto [F,A]$;
    \item on mapping complexes, the component $[F,G] \rightarrow \left[[G,A],[F,A]\right]$ of $\vee_A$ is given by
        \begin{equation*}
            f^{\vee_A} = (g \mapsto (-1)^{|f||g|}gf)
        \end{equation*}
        for homogeneous $f \in [F,G]$ and $g \in [G,A]$; and
    \item the canonical double dual identification $\can^A_F: F \rightarrow [[F,A],A]$ on an object $F$ is given by
        \begin{equation*}
            \can^A_F(x) = (f \mapsto (-1)^{|x||f|}f(x)).
        \end{equation*}
\end{enumerate}
The category $\Perf(X)$ equipped with the duality defined by $A$ is denoted 
\begin{equation*}
    (\Perf(X), \vee_A, \can^A),
\end{equation*}
or shortly $\Perf(X)^{[A]}$.
When $A$ is understood, it might be dropped from the notation, especially in the case where $A$ is the monoidal unit $\OO$.
For $n \in \Z$, $\Perf(X)^{[n]}$ abbreviates $\Perf(X)^{[\OO[n]]}$.
\end{definition}

Note that, for any $A \in \Perf(X)$, one also obtains the dg category with duality $(\Perf(X), \vee_A, -\can^A)$, the duality of which is sometimes called the \emph{skew-duality} defined by $A$. 

The canonical double dual identification $\can^A$ is in general \emph{not} a natural isomorphism, except if $A$ is an invertible complex, that is, if $A = \mc{L}[n]$ is the shift of a line bundle $\mc{L}$ on $X$.
All the dualities considered here will be of this type.
For a line bundle $\mc{L}$ on $X$, the duality $\vee_{\mc{L}[n]}$ on $\Perf(X)$ is closely related to the duality $\vee_{\mc{L}}$ on $\Vect(X)$, given by
\begin{equation*}
    \mc{F}^{\vee_{\mc{L}}} = \sheafhom(\mc{F}, \mc{L}),
\end{equation*}
whose canonical double dual identification is the evaluation map $\ev_{\mc{F}}: \F \rightarrow \F^{\vee_{\mc{L}}\vee_{\mc{L}}}$. 
Fix a line bundle $\mc{L}$ on $X$, let $n \in \Z$ and $L = \mc{L}[n]$, and equip $\Perf(X)$ and $\Vect(X)$ with the dualities $\vee_{L}$ and $\vee = \vee_{\mc{L}}$, respectively.
Let $A$ be an object of $\Perf(X)$ with differential $d$.
Let $B$ be another object of $\Perf(X)$ and let $f \in [A,B]_j$.
Definition \ref{def:dualityonperf} forces:
\begin{equation} \label{equation:dualityidentities}
    \begin{aligned}
        & (A^{\vee_{L}})_i  && = && A_{-i-n}^{\vee} \\
        & (d^{\vee_{L}})_i  && = && (-1)^{i+1}d_{-i-1-n}^{\vee} \\
        & (f^{\vee_{L}})_i  && = && (-1)^{ij} f^{\vee}_{-i-j-n} \\
        & (\can^{L}_{A})_i  && = && (-1)^{i(n+1)}\ev_{A_i}.
    \end{aligned}
\end{equation}
These identities reveal how the duality on $\Perf(X)$ relates to the perhaps more familiar duality on $\Vect(X)$. 

The following lemma from algebraic geometry, \cite[proposition 7.7]{goertz10}, plays an important role in the theory of duality on $\Perf(X)$.
Indeed, corollary \ref{corollary:shifteddualitynaturalisomorphism} is crucial in the proof of proposition \ref{proposition:tensoringbysymmetricformisformfunctor}, which shows that tensoring by a symmetric form in $\Perf(X)^{[L_1]}$ preserves duality.

\begin{lemma} \label{lemma:sheafhomcommuteswithtensorinsecondvariable}
Let $\F$, $\G$ and $\mc{H}$ be $\OO$-modules. 
If $\F$ or $\mc{H}$ is finite locally free, then the canonical map
\begin{equation*}
    \phi: \G \otimes_{\OO} \sheafhom(\F, \mc{H}) \rightarrow \sheafhom(\F, \G \otimes_{\OO} \mc{H})
\end{equation*}
given by $s \otimes f \mapsto (t \mapsto s \otimes f(t))$ is an isomorphism.
\end{lemma}

\begin{corollary} \label{corollary:tensorcommuteswithhominsecondvariableonperf}
Let $M$ and $N$ be in $\Perf(X)$, $n \in \Z$, and $\mc{L}$ a line bundle on $X$.
Then there is a natural isomorphism
\begin{equation*}
    \phi: N \otimes [M, \mc{L}[n]] \longrightarrow [M, N \otimes \mc{L}[n]]
\end{equation*}
given by $s \otimes f \mapsto (x \mapsto s \otimes f(x))$.
\end{corollary}

\begin{proof}
Note that $\phi$ is the composition
\begin{equation*}
    \begin{tikzcd}
        {N \otimes [M, \mc{L}[n]]} \arrow[r, "\nabla \otimes 1"]
            & {[\mc{L}[n], N \otimes \mc{L}[n]]} \otimes [M, \mc{L}[n]] \arrow[r, "\circ"]
                & {[M, N \otimes \mc{L}[n]]},
    \end{tikzcd}
\end{equation*}
and therefore a natural morphism.
Here, $\nabla: N \rightarrow [\mc{L}[n], N \otimes \mc{L}[n]]$ is the unit of the tensor-hom adjunction for $\mc{L}[n]$. 
Each component
\begin{equation*}
    \phi_i: (N \otimes [M, \mc{L}[n]])_i \longrightarrow [M, N \otimes \mc{L}[n]]_i
\end{equation*}
of $\phi$ is a map
\begin{equation*}
    \phi_i: \bigoplus_p N_{i-p} \otimes \sheafhom(M_{-p-n}, \mc{L}) \longrightarrow \bigoplus_q \sheafhom(M_{q-n}, N_{i+q} \otimes \mc{L})  
\end{equation*}
given by $s \otimes f \mapsto (x \mapsto s \otimes f(x))$. 
Thus $\phi_i$ is a direct sum of isomorphisms as in lemma \ref{lemma:sheafhomcommuteswithtensorinsecondvariable},
which ultimately yields that $\phi$ is a natural isomorphism.
\end{proof}

Let $\mc{L}_1$ and $\mc{L}_2$ be line bundles on $X$ and $m,n \in \Z$.
Set $L_1 = \mc{L}_1[m]$, $L_2 = \mc{L}_2[n]$ and $L_1L_2 = (\mc{L}_1 \otimes \mc{L}_2)[m+n]$, and denote simple tensors of elements similarly.

\begin{remark} \label{remark:naturalisomorphismtensorproductofshiftedlinebundles}
The map
\begin{equation*}
    \OO[m] \otimes \mc{L}_1 \otimes \OO[n] \otimes \mc{L}_2 \stackrel{1 \otimes \tau \otimes 1}{\longrightarrow} \OO[m] \otimes \OO[n] \otimes \mc{L}_1 \otimes \mc{L}_2,
\end{equation*}
where $\tau: \mc{L}_1 \otimes \OO[n] \ra \OO[n] \otimes \mc{L}_1$ is the twist map,
induces a natural isomorphism $L_1 \otimes L_2 \rightarrow L_1L_2$ given by 
\begin{equation*}
    1_{-m} \otimes x \otimes 1_{-n} \otimes y \longmapsto 1_{-m-n} \otimes xy,
\end{equation*}
where, for $i \in \Z$, $1_{-i} \in \OO[i]_{-i}$ is the multiplicative unit.
\end{remark}

\begin{corollary} \label{corollary:shifteddualitynaturalisomorphism}
Let $M,N \in \Perf(X)$.
Then there is natural isomorphism
\begin{equation*}
    \phi: [M,L_1] \otimes [N, L_2] \rightarrow [M \otimes N, L_1L_2]
\end{equation*}
given by $\phi(f \otimes g)(x \otimes y) = (-1)^{|x||g|}(f(x) \otimes g(y))$.
\end{corollary}

\begin{proof}
The map $\phi$ is the composition
\begin{equation*}
    \begin{tikzcd}[column sep=small]
        \left[M, L_1\right] \otimes \left[N, L_2\right] \arrow[r, "\tau"]
            & \left[N, L_2\right] \otimes \left[M, L_1\right] \arrow[r, "\alpha"]
                & \left[M, \left[N, L_2\right] \otimes L_1\right] \arrow[dll, "{[1,\tau]}" description] \\
        \left[M, L_1 \otimes \left[N, L_2\right]\right] \arrow[r, swap, "\beta"]
            & \left[M, \left[N, L_1 \otimes L_2\right]\right] \arrow[r]
                & \left[M \otimes N, L_1L_2\right],
    \end{tikzcd}
\end{equation*}
where $\alpha$ and $\beta$ are natural isomorphisms as in corollary \ref{corollary:tensorcommuteswithhominsecondvariableonperf}, and the final map is the tensor-hom adjunction combined with the natural isomorphism of remark \ref{remark:naturalisomorphismtensorproductofshiftedlinebundles}.
\end{proof}

Essentially, the above corollary states that
\begin{equation*}
    M^{\vee_{L_1}} \otimes N^{\vee_{L_2}} \cong (M \otimes N)^{\vee_{L_1L_2}},
\end{equation*}
which shows that tensor products are well-behaved with respect to duality, and that it should be possible to mix dualities on $\Perf(X)$.
In order to make this precise, consider the following definition.

\begin{definition} \label{def:dgformfunctor}
Let $(\mc{A}, \vee, \can)$ and $(\mc{B}, \vee, \can)$ be dg categories with duality.
A \emph{dg form functor $(F, \eta): \mc{A} \ra \mc{B}$} is a functor $F: \mc{A} \ra \mc{B}$ together with a natural transformation $\eta: F \circ \vee \ra \vee \circ F^{\op}$, whose components are called \emph{duality compatibility morphisms}, such that $\eta_A^{\vee}\can_{F(A)} = \eta_{A^{\vee}} F(\can_A)$ for all objects $A$ of $\mc{A}$.
\end{definition}

A dg form functor preserves duality, similar to the way in which exact functors preserve exactness.
They are the dg analogue of \emph{duality preserving functors} between categories with duality.
Specifically, dg form functors preserve \emph{symmetric forms}.

\begin{definition} \label{def:symmetricform}
Let $(\mc{A}, \vee, \can)$ be a dg category with duality.
A \emph{symmetric form} in $\mc{A}$ is a morphism $\phi: A \ra A^{\vee}$ such that the diagram
\begin{equation*}
    \begin{tikzcd}
        A \arrow[r, "\phi"] \arrow[d, swap, "\can"] 
            & A^{\vee} \\
        A^{\vee\vee} \arrow[ur, swap, "\phi^{\vee}"]
    \end{tikzcd}
\end{equation*}
commutes.
\end{definition}

\begin{proposition} \label{proposition:dgformfunctorspreservesymmetricforms}
Let $(\mc{A}, \vee, \can)$ and $(\mc{B}, \vee, \can)$ be dg categories with duality and $(F,\eta): \mc{A} \ra \mc{B}$ a dg form functor.
If $\phi: A \ra A^{\vee}$ is a symmetric form, then $\eta_A F(\phi): F(A) \ra F(A)^{\vee}$ is a symmetric form.
\end{proposition}

\begin{proof}
The diagram
\begin{equation*}
    \begin{tikzcd}
        F(A) \arrow[r, "F(\can_A)"] \arrow[d, swap, "\can_{F(A)}"]
            & F(A^{\vee\vee}) \arrow[r, "F(\phi^{\vee})"] \arrow[d, "\eta_{A^{\vee}}"]
                & F(A^{\vee}) \arrow[d, "\eta_A"] \\
        F(A)^{\vee\vee} \arrow[r, "\eta_{A}^{\vee}"]
            & F(A^{\vee})^{\vee} \arrow[r, "F(\phi)^{\vee}"]  
                & F(A)^{\vee}
    \end{tikzcd}
\end{equation*}
commutes, because $(F, \eta)$ is a dg form functor.
Furthermore, as $\phi$ is a symmetric form, $F(\phi^{\vee})F(\can_A) = F(\phi^{\vee} \can_A) = F(\phi)$.
Hence
\begin{equation*}
    \left(\eta_A F(\phi)\right)^{\vee} \can_{F(A)} = F(\phi)^{\vee} \eta_A^{\vee} \can_{F(A)} = \eta_A F(\phi),
\end{equation*}
as was to be shown.
\end{proof}

The following proposition provides a useful tool for constructing dg form functors (cf. \cite[remark 1.32]{schlichting17}).
Indeed, most of the form functors considered in the rest of this paper will be of this type.

\begin{proposition} \label{proposition:tensoringbysymmetricformisformfunctor}
Let $\phi: M \rightarrow [M, L_1]$ be a symmetric form in $\Perf(X)^{[L_1]}$.
Then tensoring by $M$ defines a dg form functor
\begin{equation*}
    (M, \phi)\otimes - : \Perf(X)^{[L_2]} \longrightarrow \Perf(X)^{[L_1L_2]}
\end{equation*}
with duality compatibility morphisms
\begin{equation*}
    \eta_{N}: M \otimes [N, L_2] \longrightarrow [M, L_1] \otimes [N, L_2] \longrightarrow [M \otimes N, L_1L_2],
\end{equation*}
where the first arrow is $\phi \otimes \id$ and the second arrow is the natural isomorphism of corollary \ref{corollary:shifteddualitynaturalisomorphism}.
\end{proposition}

\begin{proof}
The assignment $N \mapsto M \otimes N$ certainly defines a functor, which leaves to be shown that it defines a form functor with the provided compatibility morphisms, or equivalently, that the square
\begin{equation} \label{diagram:naturalitysquaredualitycompatibilitytensorsymmetricform}
    \begin{tikzcd}[row sep=large]
        M \otimes N \arrow[r, "\can^{L_1L_2}"] \arrow[d, "1 \otimes \can^{L_2}"]
            & \left[\left[M \otimes N, L_1L_2\right], L_1L_2\right] \arrow[d, "\eta_N^{\vee_{L_1L_2}}"] \\
        M \otimes \left[\left[N, L_2\right],L_2\right] \arrow[r, "\eta_{[N,L_2]}"]
            & \left[M \otimes \left[N, L_2\right], L_1L_2\right]
    \end{tikzcd}
\end{equation}
commutes for all $N$ in $\Perf(X)^{[L_2]}$; this amounts to a computation using the definitions of all the morphisms, and is left as an exercise.
\end{proof}

\begin{remark} \label{remark:symmetricbilinearformstosymmetricforms}
Symmetric bilinear forms correspond to symmetric forms via the tensor-hom adjunction.
Given a symmetric bilinear form $\phi: M \otimes M \ra N$, the corresponding symmetric form $\tilde{\phi}: M \ra [M,N]$ is the composition
\begin{equation*}
    M \xrightarrow{\nabla} [M, M \otimes M] \xrightarrow{[1,\phi]} [M,N],
\end{equation*}
so that $\tilde{\phi}(x) = (y \mapsto \phi(x \otimes y))$. 
\end{remark}

One consequence of proposition \ref{proposition:tensoringbysymmetricformisformfunctor} is that skew-symmetric forms can be transformed into symmetric forms.

\begin{remark} \label{remark:skewsymmetrictosymmetric}
For the purpose of this remark, let $L = \mc{L}[m]$, where $\mc{L}$ is a line bundle on $X$ and $m \in \Z$, and let $\epsilon = \pm 1$ and $i \in \Z$.
Consider a symmetric bilinear form $\phi: M \otimes M \rightarrow L$ in $(\Perf(X), \vee_L, \epsilon \can^L)$. 
Then $\phi \tau = \epsilon \phi$. 
The multiplication map $\mu: \OO[i] \otimes \OO[i] \rightarrow \OO[2i]$ given by $x \otimes y \mapsto xy$ satisfies $\mu \tau = (-1)^i\mu$, as witnessed by the identity 
\begin{equation*}
    \mu(\tau(x \otimes y)) = (-1)^{i^2}yx = (-1)^i xy = (-1)^i\mu(x \otimes y).    
\end{equation*}
Thus $\mu\tau \otimes \phi\tau = (-1)^i\epsilon(\mu \otimes \phi)$. 
Since the diagram
\begin{equation*}
    \begin{tikzcd}
        {\OO[1] \otimes M \otimes \OO[1] \otimes M} \arrow[r, "1 \otimes \tau \otimes 1"] \arrow[d, "\tau"]
            & \OO[1] \otimes \OO[1] \otimes M \otimes M \arrow[d, "\tau \otimes \tau"] \\
        \OO[1] \otimes M \otimes \OO[1] \otimes M \arrow[r, "1 \otimes \tau \otimes 1"]
            & \OO[1] \otimes \OO[1] \otimes M \otimes M
    \end{tikzcd}
\end{equation*}
commutes, as can be seen from a direct computation, it follows that the composition
\begin{equation*}
    \begin{tikzcd}
        \psi:
            &[-3em] \OO[1] \otimes M \otimes \OO[1] \otimes M \arrow[r,"1 \otimes \tau \otimes 1"]
                & \OO[1] \otimes \OO[1] \otimes M \otimes M \arrow[r, "\mu \otimes \phi"]
                    & \OO[2] \otimes L
    \end{tikzcd}
\end{equation*}
satisfies $\psi \tau = (-1)^i\epsilon\psi$. 
This yields an equivalence of dg categories with duality
\begin{equation*}
    (\OO[i],\mu)\otimes - : (\Perf(X), \vee_L, \epsilon\can^L) \longrightarrow (\Perf(X),  \vee_{L[2i]}, (-1)^i\epsilon\can^{L[2i]}),
\end{equation*}
which in particular shows how to turn skew-symmetric forms into symmetric ones by taking $\epsilon = -1$ and $i = \pm 1$. 
\end{remark}

Another consequence of proposition \ref{proposition:tensoringbysymmetricformisformfunctor} is that the tensor product of two symmetric bilinear forms is another symmetric bilinear form.

\begin{corollary} \label{corollary:tensorproductofsymmetricformsissymmetric}
Let $\phi: M \otimes M \rightarrow L_1$ and $\psi: N \otimes N \rightarrow L_2$ be symmetric bilinear forms in $\Perf(X)^{[L_1]}$ and $\Perf(X)^{[L_2]}$, respectively.
Then the composition
\begin{equation*}
    \begin{tikzcd}
        M \otimes N \otimes M \otimes N \arrow[r, "1 \otimes \tau \otimes 1"]
            & M \otimes M \otimes N \otimes N \arrow[r, "\phi \otimes \psi"]
                & L_1L_2
    \end{tikzcd}
\end{equation*}
is a symmetric bilinear form in $\Perf(X)^{[L_1L_2]}$. 
\end{corollary}

\begin{proof}
The bilinear forms $\phi$ and $\psi$ define symmetric forms $\tilde{\phi}: M \ra [M,L_1]$ and $\tilde{\psi}: N \ra [N,L_2]$. 
Let $F = (M, \tilde{\phi}) \otimes -$ be the form functor of proposition \ref{proposition:tensoringbysymmetricformisformfunctor} with duality compatibility morphisms $\eta_A$.
Then $\eta_N F(\tilde{\psi})$ is a symmetric form by proposition \ref{proposition:dgformfunctorspreservesymmetricforms}, and the symmetric bilinear form it induces is precisely that of the statement of the lemma.
\end{proof}

\begin{corollary} \label{corollary:equivalencesquareoflinebundle}
Let $L = \mc{L}[m]$, where $\mc{L}$ is a line bundle on $X$ and $m \in \Z$, equipped with the trivial symmetric form $\mu: L \otimes L \ra L^{\otimes 2}$
Then tensoring by $L$ induces an equivalence
\begin{equation*}
    (L, \mu) \otimes -: \Perf(X)^{[0]} \longrightarrow \Perf(X)^{[L^{\otimes 2}]}
\end{equation*}
of pretriangulated dg categories with duality.
\end{corollary}

\begin{proof}
It is immediate from proposition \ref{proposition:tensoringbysymmetricformisformfunctor} that tensoring by $L$ gives a dg form functor
\begin{equation*}
    F: \Perf(X)^{[0]} \longrightarrow \Perf(X)^{[L^{\otimes 2}]}
\end{equation*}
whose duality compatibility morphisms are isomorphisms.
Furthermore, tensoring by $L$ is an equivalence; an inverse is given by tensoring with $[L, \OO]$. 
It follows that $F$ is an equivalence of pretriangulated dg categories with duality.
\end{proof}

Thus concludes this demonstration of the general yoga of dualities on $\Perf(X)$, some of the main postures being:
\begin{enumerate}[label=(\roman*)]
    \item the identities (\ref{equation:dualityidentities}) relating duality on $\Perf(X)$ to duality on $\Vect(X)$;
    \item remark \ref{remark:skewsymmetrictosymmetric}, transmutating skew-symmetric forms into symmetric ones by passing to a different duality on $\Perf(X)$;
    \item corollary \ref{corollary:tensorproductofsymmetricformsissymmetric}, stating that the product of symmetric forms is a symmetric form; and
    \item corollary \ref{corollary:equivalencesquareoflinebundle}, shifting a duality on $\Perf(X)$ by the square $L^{\otimes 2}$ of an invertible complex $L$ without essentially changing it.
\end{enumerate}

\subsection{Constructing symmetric forms from Koszul complexes}
    \label{subsection:constructingsymmetricformsfromkoszulcomplexes}

This section applies the results from the previous section to the specific case of projective bundles, and lays the groundwork for the proof of the projective bundle formula.    

Let $X$ be a scheme satisfying the resolution property and let $\mc{E}$ be a finite locally free $\OO_X$-module of rank $r+1$.
Let $\P = \P(\mc{E})$ be the projective bundle over $X$ associated to $\mc{E}$, with projection map $\pi: \P \rightarrow X$.
Set $s = \lceil r/2 \rceil$ and $\OO = \OO_{\P}$ for ease of notation.
By the construction of the projective bundle, there is a canonical surjection $\pi^*\mc{E} \twoheadrightarrow \OO(1)$, giving rise to the Koszul complex $K$
\begin{equation*}
    K: \quad 
        0 \rightarrow
        \ldots \rightarrow
        \Lambda^{i}\pi^*\mc{E}(-i) \rightarrow
        \ldots \rightarrow
        \Lambda^1 \pi^*\mc{E}(-1) \rightarrow
        \OO \rightarrow
        0
\end{equation*}
with $K_{-i} = \Lambda^i \pi^*\mc{E} \otimes \OO(-i)$ in cohomological degree $-i$, which is acyclic by \cite[section 4.6]{thomason90}.
It also holds that $\Lambda^{r+1}\pi^*\mc{E} \cong \pi^*\det \mc{E}$ and $\Lambda^1 \pi^*\mc{E} = \pi^*\mc{E}$. 
View $K$ as a differential graded algebra with $\Lambda^1\pi^*\mc{E}(-1)$ in degree $-1$, so that the cohomological degree and the degree of the grading coincide, and write $|x|$ for the degree of a homogeneous element $x \in K$. 
For ease of notation, fix $\Delta = \det \pi^*\mc{E}(-r-1)$. 

The rest of this section is dedicated to the construction of a symmetric form $(H,\psi)$ such that the cone of $\psi$ is the Koszul complex, which can be seen as a generalization of \cite[section 4]{balmer05koszul}.
Two cases are distinguished, $r$ is even and $r$ is odd, because the construction in the first case is easier than that in the second case.
Such a symmetric form $(H, \psi)$ is a quasi-isomorphism in $\Perf(\P)$ because the Koszul complex is acyclic, and therefore defines an element of the Grothendieck-Witt group $\GW^{[r]}_0(\P, \Delta)$, which is a key ingredient in the proof of theorem \ref{theorem:projectivebundleformulaforgwwithtwistofdifferentparity}.

The following proposition is an adaptation of the well-known and useful fact that the Koszul complex is self-dual.

\begin{proposition} \label{proposition:koszulisselfdual}
The bilinear form $\mu: K \otimes K \rightarrow \Delta[r+1]$ given by the composition
\begin{equation*}
    K \otimes K \xrightarrow{\wedge} K \xrightarrow{\pr} \Delta[r+1],
\end{equation*}
where the last map is the projection map, is symmetric and non-degenerate.
\end{proposition}

\begin{proof}
Symmetry can be checked by a direct computation involving the graded commutativity of $K$ and the sign change on the twist map $\tau: K \otimes K \rightarrow K \otimes K$.

Non-degeneracy holds because the components of the induced symmetric form $\tilde{\mu}: K \ra [K, \Delta[r+1]]$ are locally isomorphisms of free $\OO$-modules and therefore isomorphisms. 
\end{proof}

Fix an integer $\ell$ such that $-r-1 \leq \ell \leq -1$ and let $M = K_{\leq \ell}$ be the naive truncation of $K$:
\begin{equation*}
    \begin{tikzcd}[column sep=small]
        M \arrow[d, "\iota"]:
            & K_{-r-1} \arrow[r, "d"] \arrow[d, equal]
                & K_{-r} \arrow[r, "d"] \arrow[d, equal]
                    & \dots \arrow[r, "d"]
                        & K_{\ell} \arrow[r] \arrow[d, equal]
                            & 0 \arrow[r] \arrow[d]
                                & \dots \arrow[r]
                                    & 0 \arrow[d] \\
        K:
            & K_{-r-1} \arrow[r, "d"]
                & K_{-r} \arrow[r, "d"]
                    & \dots \arrow[r, "d"]
                        & K_{\ell} \arrow[r, "d"]
                            & K_{\ell+1} \arrow[r, "d"]
                                & \dots \arrow[r, "d"]
                                    & K_0.
    \end{tikzcd}
\end{equation*}
Since $K$ is a differential graded $\OO$-algebra, there is a multiplication map $\wedge: K \otimes K \rightarrow K$ given by the wedge product.
Let $\varphi$ be the composition
\begin{equation*}
    \begin{tikzcd}
        M \otimes M \arrow[r, "d \iota \otimes \iota"]
            & K \otimes K \arrow[r, "\mu"]
                & \Delta[r+2],
    \end{tikzcd}
\end{equation*}
where $\mu$ is the symmetric bilinear form of proposition \ref{proposition:koszulisselfdual}.
In a formula, 
\begin{equation*}
\varphi(x \otimes y) = 
\left\{
\begin{array}{ll}
     d(x) \wedge y & \tn{if }|x| + |y| = -r-2 \\
     0 & \tn{otherwise.} 
\end{array}    
\right.
\end{equation*}
Now $\varphi$ is a symmetric bilinear form, which will be molded in such a way that its cone becomes the Koszul complex.

\begin{proposition} \label{proposition:symmetricformkoszulcomplexprojectivebundle}
The map
\begin{equation*}
    \varphi: M \otimes M \longrightarrow \Delta[r+2]
\end{equation*}
defines a symmetric form $\phi: M[-1] \ra [M[-1], \Delta[r]]$ in $\Perf(\P)^{[\Delta[r]]}$ given by $x \mapsto (y \mapsto (-1)^{|x|}\varphi(x \otimes y))$. 
\end{proposition}

\begin{proof}
Let $\tau: M \otimes M \rightarrow M \otimes M$ be the switch map $x \otimes y \mapsto (-1)^{|x||y|}y \otimes x$.
By \cite[remark 1.31]{schlichting17}, $\varphi$ defines a skew-symmetric form $\phi: M \rightarrow [M, \Delta[r+2]]$ if $\varphi$ satisfies $\varphi = -\varphi \tau$, which is what will be shown.
By construction, 
\begin{equation*}
    \varphi(\tau(x \otimes y)) = 0 = -\varphi(x \otimes y)
\end{equation*}
if $|x| + |y| \neq -r-2$.
Therefore, let $x \otimes y \in M \otimes M$ with $|x| + |y| = -r-2$.
Then a direct computation involving the graded Leibniz rule shows that
\begin{equation*}
    \varphi(\tau(x \otimes y)) = (-1)^{|y| + 1}(-1)^{|y|} d(x) \wedge y = -d(x) \wedge y = -\varphi(x \otimes y).
\end{equation*}
By remark \ref{remark:skewsymmetrictosymmetric}, tensoring $\varphi$ with the skew-symmetric form $\OO[-1] \otimes \OO[-1] \rightarrow \OO[-2]$ yields the desired symmetric form $\phi$. 
\end{proof}

For the remainder of this section, let $\phi$ be the symmetric form of proposition \ref{proposition:symmetricformkoszulcomplexprojectivebundle}.
Note that $\phi$ is not necessarily a quasi-isomorphism in $\Perf(\P)$ and therefore does not necessarily define an element of $\GW^{[r]}(\P, \Delta)$. 
However, the ideas of \cite[section 4]{balmer05koszul} can be combined with the technique of the proof of \cite[theorem 1.5]{walter03projective} to obtain a symmetric form $(H,\psi)$ having the Koszul complex, or a related acyclic complex, as its cone.

\begin{proposition}\label{proposition:evenprojectivebundlecutskoszulintwo}
Assume that $r$ is even and let $(H, \psi) = (M, \phi)$ with $\ell = -s-2$.
Then $\psi$ is a quasi-isomorphism. 
\end{proposition}

\begin{proof}
By proposition \ref{proposition:koszulisselfdual}, $\psi$ becomes isomorphic to the map of complexes
\begin{equation*}
    \begin{tikzcd}
        K_{-r-1} \arrow[r, "-d"]
            & \ldots \arrow[r, "-d"]
                & K_{-s-2} \arrow[d, "d"] \arrow[r]
                    & 0
                        & { } \\
        { }
            & 0 \arrow[r]
                & K_{-s-1} \arrow[r, "d"]
                    & \ldots \arrow[r, "d"]
                        & K_0,
    \end{tikzcd}
\end{equation*}
concentrated in cohomological degrees $[-r,0]$. 
Therefore, the cone $C(\psi)$ of $\psi$ is isomorphic to the Koszul complex $K$, which is acyclic.
Consequently, $\psi$ is a quasi-isomorphism, as was to be shown.
\end{proof}

Now suppose that $r$ is odd.
Taking $\ell = -s$, the cone $C(\phi)$ becomes isomorphic to the complex $K \oplus K_{-s}[s]$, where $K_{-s} = \Lambda^s \pi^* \mc{E}(-s)$ is viewed as a complex concentrated in degree zero.
The middle terms of $C(\phi)$ are
\begin{equation*}
    \begin{tikzcd}
        \ldots \arrow[r]
            & K_{-s-1} \arrow[r]
                & K_{-s}^{\oplus 2} \arrow[r]
                    & K_{-s+1} \arrow[r]
                        & \ldots,
    \end{tikzcd}
\end{equation*}
where the differential $K_{-s-1} \rightarrow K_{-s}^{\oplus 2}$ is given by $x \mapsto (-dx, dx)$ and the differential $K_{-s}^{\oplus 2} \rightarrow K_{-s+1}$ is given by $(x, y) \mapsto dx + dy$. 
The idea presents itself that there might be a symmetric form whose cone is an acyclic complex that is closely related to the Koszul complex, as long as $K_{-s}$ ``splits into two dual parts''. 
The remainder of this section is dedicated to constructing such a symmetric form, with a suitable assumption on $\Lambda^s \mc{E}$.

For the moment, consider the exact category with duality $\Vect(X)^{[\det \mc{E}]}$ of locally free sheaves on $X$, where the duality is denoted by $\natural$. 
The reason for this excursion to $\Vect(X)$ is that the symmetric form $(H,\psi)$ in $\Perf(\P)^{[\Delta[r]]}$, which is being constructed for the splitting of the homotopy fibration in theorem \ref{theorem:projectivebundleformulaforgwwithtwistofdifferentparity}(ii), needs to have a specific image in $\Perf(X)^{[0]}$. 

\begin{lemma} \label{lemma:middletermexterioralgebraissymmetric}
The restriction of the wedge product $\wedge: \Lambda \mc{E} \otimes \Lambda \mc{E} \ra \Lambda \mc{E}$ to $\Lambda^s \mc{E} \subset \Lambda \mc{E}$ induces a $(-1)^s$-symmetric form
\begin{equation*}
    \nu': \Lambda^s \mc{E} \longrightarrow (\Lambda^s \mc{E})^{\natural}
\end{equation*}
in $\Vect(X)^{[\det \mc{E}]}$, and in particular defines an element $\nu' \in W^{[r+1]}(X, \det \mc{E})$.
\end{lemma}

\begin{proof}
Since $x \wedge y = (-1)^{s^2}y \wedge x = (-1)^s y \wedge x$ for $x, y \in \Lambda^s \mc{E}$, $\nu'$ is a $(-1)^s$-symmetric form, which is also an isomorphism.
As $2s = r+1$, it follows that $\nu'$ defines an element in $W^{[r+1]}(X, \det \mc{E})$. 
\end{proof}

Now assume that the element $\nu'$ of lemma \ref{lemma:middletermexterioralgebraissymmetric} vanishes in $W^{[r+1]}(X, \det \mc{E})$.
Then $(\Lambda^s \mc{E}, \nu')$ is stably metabolic, so by \cite[remark 29]{balmer05witt} there exists a metabolic space $(\mc{N}', \sigma')$ such that $(\Lambda^s \mc{E}, \nu') \perp (\mc{N}', \sigma')$ is split metabolic and even hyperbolic because $2$ is invertible;
let $(\mc{N}', \sigma')$ be such a metabolic space and fix a split exact sequence
\begin{equation} \label{diagram:splitlagrangianexactsequencemiddleterm}
    \begin{tikzcd}[column sep=huge]
        \mc{P}' \arrow[r, hook, shift left=0.2em, "\iota_{\mc{P}'}"]
            & \Lambda^s \mc{E} \oplus \mc{N}' \arrow[r, two heads, shift left=0.2em, "\iota_{\mc{P}'}^{\natural}(\nu' \oplus \sigma')"] \arrow[l, two heads, shift left=0.2em, "\pr_{\mc{P}'}"]
                & \mc{P}'^{\natural}, \arrow[l, hook, shift left=0.2em, "(\nu' \oplus \sigma')^{-1}\pr_{\mc{P}'}^{\natural}"]
    \end{tikzcd}
\end{equation}
where $\mc{P}'$ is a split Lagrangian of $\Lambda^s \mc{E} \oplus \mc{N}'$ with orthogonal complement $\mc{P}'^{\natural}$.
It is useful to think of this construction as extending $\Lambda^s \mc{E}$ by $\mc{N}'$ such that it can be chopped into two halves $\mc{P}'$ and $\mc{P}'^{\natural}$ which are dual to each other.

Additionally, fix a short exact sequence
\begin{equation*}
    \begin{tikzcd}
        \mc{S}' \arrow[r, hook, "\alpha"]
            & \mc{N}' \arrow[r, two heads, "\alpha^{\natural}\sigma"]
                & \mc{S}'^{\natural},
    \end{tikzcd}
\end{equation*}
which exists since $\mc{N}'$ is metabolic.

Note that $\pi^*: \Vect(X)^{[\det \mc{E}]} \ra \Vect(\P)^{[\det \pi^*\mc{E}]}$ is an exact duality-preserving functor. 
Furthermore, the symmetric form $\id: \OO(-s) \ra \OO(-s)$ in $\Vect(\P)^{[\OO(-r-1)]}$ yields a duality-preserving equivalence
\begin{equation*}
    (\OO(-s),\id) \otimes - : \Vect(\P)^{[\det \pi^* \mc{E}]} \longrightarrow \Vect(\P)^{[\Delta]}.
\end{equation*}
Let $F$ be the composition
\begin{equation*}
    \begin{tikzcd}[column sep=huge]
        F:
            &[-5em] \Vect(X)^{[\det \mc{E}]} \arrow[r, "\pi^*"] 
                & \Vect(\P)^{[\det \pi^* \mc{E}]} \arrow[r, "{(\OO(-s),\id) \otimes -}"]
                    & \Vect(\P)^{[\Delta]},
    \end{tikzcd}
\end{equation*}
and let $\mc{P} = F(\mc{P}')$, $\mc{N} = F(\mc{N}')$, $\mc{S} = F(\mc{S}')$, $\nu = F(\nu')$ and $\sigma = F(\sigma')$. 
Note that $\nu$ is the $(-1)^s$-symmetric form
\begin{equation*}
    \nu: \Lambda^s \pi^*\mc{E}(-s) \longrightarrow \Lambda^s \pi^*\mc{E}^{\natural}(-s)
\end{equation*}
induced by the restriction of the wedge product $\wedge: K \otimes K \ra K$ of the Koszul complex to the middle term $K_{-s}$.
The image of the exact sequence (\ref{diagram:splitlagrangianexactsequencemiddleterm}) under $F$ is another split exact sequence
\begin{equation} \label{diagram:splitlagrangianexactsequencemiddletermkoszul}
    \begin{tikzcd}[column sep=huge]
        \mc{P} \arrow[r, hook, shift left=0.2em, "\iota_{\mc{P}}"]
            & K_{-s} \oplus \mc{N} \arrow[r, two heads, shift left=0.2em, "\iota_{\mc{P}}^{\natural}(\nu \oplus \sigma)"] \arrow[l, two heads, shift left=0.2em, "\pr_{\mc{P}}"]
                & \mc{P}^{\natural} \arrow[l, hook, shift left=0.2em, "(\nu \oplus \sigma)^{-1}\pr_{\mc{P}}^{\natural}"]
    \end{tikzcd}
\end{equation}
in $\Vect(\P)^{[\Delta]}$, where the duality on $\Vect(\P)$ induced by $\Delta$ is also denoted by $\natural$. 
The fact that $\mc{P}$, $\mc{N}$ and $\mc{S}$ lie in the image of $F$ will be crucial in the proof of theorem \ref{theorem:projectivebundleformulaforgwwithtwistofdifferentparity}(ii). 
The following two technical lemmas construct the central square of $(H, \psi)$.
The takeaway is that this square is symmetric when appropriately embedded in $\Perf(\P)$.

\begin{lemma} \label{lemma:centralsquarekoszulcutintwocommutes}
The square
\begin{equation} \label{diagram:centralsquarekoszulcutintwo}
    \begin{tikzcd}[row sep=large, column sep=large]
        K_{-s-1} \oplus \mc{S} \arrow[r, "d \oplus \alpha"] \arrow[d, swap, "d \oplus \alpha"]
            & K_{-s} \oplus \mc{N} \arrow[r, two heads, "\pr_{\mc{P}}"]
                & \mc{P} \arrow[d, hook, "(\nu \oplus \sigma)\iota_{\mc{P}}"] \\
        K_{-s} \oplus \mc{N} \arrow[d, two heads, swap, "\iota_{\mc{P}}^{\natural}(\nu \oplus \sigma)"]
            & { }
                & K_{-s}^{\natural} \oplus \mc{N}^{\natural} \arrow[d, "(d \oplus \alpha)^{\natural}"] \\
        \mc{P}^{\natural} \arrow[r, hook, swap, "\pr_{\mc{P}}^{\natural}"]
                & K_{-s}^{\natural} \oplus \mc{N}^{\natural} \arrow[r, swap, "(d \oplus \alpha)^{\natural}"]
                    & K_{-s-1}^{\natural} \oplus \mc{S}^{\natural}
    \end{tikzcd}
\end{equation}
anti-commutes.
\end{lemma}

\begin{proof}
The split exact sequence (\ref{diagram:splitlagrangianexactsequencemiddletermkoszul}) yields
\begin{equation*}
    \iota_{\mc{P}}\pr_{\mc{P}} + (\nu \oplus \sigma)^{-1}\pr_{\mc{P}}^{\natural}\iota_{\mc{P}}^{\natural}(\nu \oplus \sigma) = \id_{(\K_{-s} \oplus \mc{N})},
\end{equation*}
which becomes
\begin{equation*}
    (\nu \oplus \sigma)\iota_{\mc{P}}\pr_{\mc{P}} + \pr_{\mc{P}}^{\natural}\iota_{\mc{P}}^{\natural}(\nu \oplus \sigma) = (\nu \oplus \sigma)
\end{equation*}
when composed with $(\nu \oplus \sigma)$.
Furthermore, note that $d^{\natural}\nu d: K_{-s-1} \rightarrow K_{-s-1}^{\natural}$ is given by $x \mapsto (y \mapsto d(x) \wedge d(y))$, but $d(x) \wedge d(y) = d(x \wedge d(y)) = d(0) = 0$, so $d^{\natural}\nu d = 0$. 
Thus the sum of the two paths of the square satisfies
\begin{equation*}
    \begin{aligned}
        & (d \oplus \alpha)^{\natural}(\nu \oplus \sigma)\iota_{\mc{P}}\pr_{\mc{P}}(d \oplus \alpha) + (d \oplus \alpha)^{\natural}\pr_{\mc{P}}^{\natural}\iota_{\mc{P}}^{\natural}(\nu \oplus \sigma)(d \oplus \alpha) \\
        & \qquad = (d \oplus \alpha)^{\natural}((\nu \oplus \sigma)\iota_{\mc{P}}\pr_{\mc{P}} + \pr_{\mc{P}}^{\natural}\iota_{\mc{P}}^{\natural}(\nu \oplus \sigma))(d \oplus \alpha) \\
        & \qquad = (d \oplus \alpha)^{\natural}(\nu \oplus \sigma)(d \oplus \alpha) \\
        & \qquad = (d^{\natural}\nu d \oplus \alpha^{\natural}\sigma\alpha) \\
        & \qquad = 0,
    \end{aligned}
\end{equation*}
which proves the result.
\end{proof}

The next lemma applies the previous one in the context of $\Perf(\P)^{[\Delta[r]]}$, with the duality $\vee = \vee_{\Delta[r]}$.
It is essentially an application of the identities (\ref{equation:dualityidentities}).

\begin{lemma} \label{lemma:centralsquarekoszulintwoissymmetric}
The map of complexes $\psi$ given by
\begin{equation*}
    \begin{tikzcd}[column sep=huge, row sep=huge]
        K_{-s-1} \oplus \mc{S} \arrow[r, "\pr_{\mc{P}}(d \oplus \alpha)"] \arrow[d, swap, "\iota_{\mc{P}}^{\natural}(\nu \oplus \sigma)(d \oplus \alpha)"]
            &[1em] \mc{P} \arrow[d, "(-1)^s(d \oplus \alpha)^{\natural}(\nu \oplus \sigma)\iota_{\mc{P}}"] \\
        \mc{P}^{\natural} \arrow[r, swap, "(-1)^{s+1}(d \oplus \alpha)^{\natural}\pr_{\mc{P}}^{\natural}"]
            & K_{-s-1}^{\natural} \oplus \mc{S}^{\natural}
    \end{tikzcd}
\end{equation*}
concentrated in cohomological degrees $[-s,-s+1]$ is symmetric in the pretriangulated dg category with duality $(\Perf(\P), \vee, \can)$.
\end{lemma}

\begin{proof}
By lemma \ref{lemma:centralsquarekoszulcutintwocommutes}, the square commutes.
For notational convenience, the subscript of the evaluation map $\ev_{\mc{F}}$ is suppressed.
It remains to be shown that $\psi = \psi^{\vee}\can$.
Note that $(\psi^{\vee})_{-s+1} = \psi^{\natural}_{s-1-r} = \psi^{\natural}_{-s}$.
Since $(\nu \oplus \sigma)$ is $(-1)^s$-symmetric,
\begin{align*}
    (\iota_{\mc{P}}^{\natural}(\nu \oplus \sigma)(d \oplus \alpha))^{\natural}\ev
        & = (d \oplus \alpha)^{\natural}(\nu \oplus \sigma)^{\natural}\iota_{\mc{P}}^{\natural\natural}\ev \\
        & = (d \oplus \alpha)^{\natural}(\nu \oplus \sigma)^{\natural}\ev \iota_{\mc{P}} \\
        & = (-1)^s(d \oplus \alpha)^{\natural}(\nu \oplus \sigma)\iota_{\mc{P}}.
\end{align*}
A similar computation shows that
\begin{equation*}
    \left((-1)^s(d \oplus \alpha)^{\natural}(\nu \oplus \sigma)\iota_{\mc{P}}\right)^{\natural}\ev = \iota_{\mc{P}}^{\natural}(\nu \oplus \sigma)(d \oplus \alpha),
\end{equation*}
which concludes the proof.
\end{proof}

Let $H$ be the following complex, concentrated in degrees $[-r,-s+1]$:
\begin{equation*}
    \begin{tikzcd}
        K_{-r-1} \arrow[r, "d"]
            &[-0.5em] \dots \arrow[r, "d"]
                &[-0.5em] K_{-s-2} \arrow[r, "d"]
                    & K_{-s-1} \oplus \mc{S} \arrow[r, "\pr_{\mc{P}}(d \oplus \alpha)"]
                        &[1em] \mc{P},
    \end{tikzcd}
\end{equation*}
where $d: K_{-s-2} \rightarrow K_{-s-1} \oplus \mc{S}$ is the composition of the differential $d: K_{-s-2} \rightarrow K_{-s-1}$ and the canonical inclusion $K_{-s-1} \rightarrow K_{-s-1} \oplus \mc{S}$.
In some sense, $H$ is ``half'' of the Koszul complex, with some surgical alterations at the end to ease the conditions under which it can be constructed.
Piecing together the various results obtained thus far yields the following theorem.

\begin{theorem} \label{theorem:middletermofkoszulcomplexwittnilcutskoszulintwo}
Assume that $\nu' = 0$ in $W^{[r+1]}(X, \det \mc{E})$.
Let $\psi: H \rightarrow H^{\vee}$ be the chain map in $\Perf(\P)^{[\Delta[r]]}$ given by
\begin{equation*}
    \begin{tikzcd}[column sep=small]
        H_{-r} \arrow[r] \arrow[d]
            & \dots \arrow[r] \arrow[d]
                & H_{-s-1} \arrow[r] \arrow[d]
                    & H_{-s} \arrow[r] \arrow[d]
                        & H_{-s+1} \arrow[r] \arrow[d]
                            & 0 \arrow[r] \arrow[d]
                                & \dots \arrow[r] \arrow[d]
                                    & 0 \arrow[d] \\
        0 \arrow[r]
            & \dots \arrow[r]
                & 0 \arrow[r]
                    & H_{-s+1}^{\natural} \arrow[r]
                        & H_{-s}^{\natural} \arrow[r]
                            & H_{-s-1}^{\natural} \arrow[r]
                                & \dots \arrow[r]
                                    & H_{-r}^{\natural},
    \end{tikzcd}
\end{equation*}
where the central square is that of lemma \ref{lemma:centralsquarekoszulintwoissymmetric}.
Then $\psi$ is symmetric and a quasi-isomorphism.
\end{theorem}

\begin{proof}
In this proof, let $d'$ denote the differential of $H$. 
The central square commutes and is symmetric by lemma $\ref{lemma:centralsquarekoszulintwoissymmetric}$. 
The square directly left of the central square commutes since
\begin{equation*}
    \psi_{-s}d'_{-s-1} = \iota_{\mc{P}}^{\natural}(\nu \oplus \sigma)(d \oplus \alpha)d'_{-s-1} = 0,
\end{equation*}
and similarly for the square directly right of the central square.
It follows that $\psi$ is symmetric and it remains to be shown that $\psi$ is a quasi-isomorphism, or equivalently, that the cone $C(\psi)$ of $\psi$ is acyclic.
Note that $C(\psi)$ is the complex
\begin{equation*}
    \begin{tikzcd}[column sep=tiny]
        K_{-r-1} \arrow[r]
            & \dots \arrow[r]
                & K_{-s-1} \oplus \mc{S} \arrow[r]
                    & K_{-s} \oplus \mc{N} \arrow[r]
                        & K_{-s-1}^{\natural} \oplus \mc{S}^{\natural} \arrow[r]
                            & \dots \arrow[r]
                                & K_{-r-1}^{\natural},
    \end{tikzcd}
\end{equation*}
which is isomorphic to the Koszul complex away from the middle degrees $[-s-2, -s+2]$ by proposition \ref{proposition:koszulisselfdual}.
In the middle degrees, $C(\phi)$ is given as
\begin{equation*}
    \begin{tikzcd}
        K_{-s-2} \arrow[r, "d"]
            &[-0.75em] K_{-s-1} \oplus \mc{S} \arrow[r, "d \oplus \alpha"]
                & K_{-s} \oplus \mc{N} \arrow[r, "d \oplus \alpha^{\natural}\sigma"]
                    & K_{-s-1}^{\natural} \oplus \mc{S}^{\natural} \arrow[r, "d"]
                        &[-0.75em] K_{-s-2}^{\natural}.
    \end{tikzcd}
\end{equation*}
Thus $C(\phi)$ is the direct sum of the acyclic koszul complex $K$ and the exact sequence $\mc{S} \rightarrow \mc{N} \rightarrow \mc{S}^{\natural}$, seen as an acyclic complex concentrated in degrees $[-s-1,-s+1]$.
Therefore, $C(\phi)$ itself is acyclic, as was to be shown.
\end{proof}

Theorem \ref{theorem:middletermofkoszulcomplexwittnilcutskoszulintwo} finishes the construction of the symmetric form $(H,\psi)$, but it is not a priori clear when the condition on $\nu'$ holds.
The following lemma provides a useful criterion.

\begin{lemma} \label{lemma:quotientbundleoddrankimplieswittnil}
If $\mc{E}$ admits a quotient bundle of odd rank, then $\nu$ vanishes in $W^{[r+1]}(X, \det \mc{E})$.
\end{lemma}

\begin{proof}
This is \cite[proposition 8.1]{walter03projective}.
Note that if $\mc{E}$ is trivial, it certainly admits a quotient bundle of odd rank.
\end{proof}

\subsection{Localization and additivity for Grothendieck-Witt spectra}
    \label{subsection:localizationandadditivityforgrothendieckwittspectra}

This section is an intermezzo recalling two fundamental results of Grothendieck-Witt theory as presented in \cite{schlichting17}, which are adapted to the current context.
They are key ingredients in the proof of main theorems \ref{maintheorem:projectivebundleformulawithtwistofsameparityasdimension} and \ref{maintheorem:projectivebundleformulaforgwwithtwistofdifferentparity}.
The first is a combination of \cite[theorem 6.6]{schlichting17} and \cite[theorem 8.10]{schlichting17}.

\begin{theorem}[localization for $\GW$ and $\GGW$] \label{theorem:localizationforgw}
Let $\mc{A} \rightarrow \mc{B} \rightarrow \mc{C}$ be a Morita exact sequence of dg categories with weak equivalences and duality, all containing $\tfrac{1}{2}$.
Let $n \in \Z$.
Then there is a homotopy fibration of Karoubi-Grothendieck-Witt spectra
\begin{equation*}
    \begin{tikzcd}
        \GGW^{[n]}(\mc{A}) \arrow[r]
            & \GGW^{[n]}(\mc{B}) \arrow[r]
                & \GGW^{[n]}(\mc{C}).
    \end{tikzcd}
\end{equation*}
If, moreover, the sequence $\mc{A} \rightarrow \mc{B} \rightarrow \mc{C}$ of dg categories is quasi-exact, then there is also a homotopy fibration of Grothendieck-Witt spectra
\begin{equation*}
    \begin{tikzcd}
        \GW^{[n]}(\mc{A}) \arrow[r]
            & \GW^{[n]}(\mc{B}) \arrow[r]
                & \GW^{[n]}(\mc{C}).
    \end{tikzcd}
\end{equation*}
\end{theorem}

The second result is a general version of additivity for GW-theory, c.f. \cite[theorem 2.5]{walter03projective}, which allows for semi-orthogonal decompositions of arbitrary size.

\begin{theorem}[additivity for $\GW$] \label{theorem:generaladditivityforgw}
Let $(\mc{A},w,\vee)$ be a pretriangulated dg category with weak equivalences and duality equipped with a semi-orthogonal decomposition $\langle \mc{A}_{0}, \mc{A}_{1}, \dots, \mc{A}_r \rangle$, such that $\mc{A}_{i}^{\vee} \subset \mc{A}_{r-i}$.
\begin{enumerate}[label=(\roman*)]
    \item The duality $\vee: \mc{A}^{\op} \rightarrow \mc{A}$ induces equivalences $\mc{A}_i^{\op} \simeq \mc{A}_{r-i}$
    \item Suppose $r$ is odd. 
    Let $q = (r-1)/2$.
    Then the functor
    \begin{equation*}
        \begin{aligned}
            \prod\limits_{i = 0}^q H\mc{A}_i    & \longrightarrow   \mc{A} \\
            \prod\limits_{i = 0}^q (A_i, B_i)   & \longmapsto       \bigoplus\limits_{i=0}^q A_i \oplus B_i^{\vee},
        \end{aligned}
    \end{equation*}
    where $H\mc{A}_i$ is the hyperbolic category with duality associated to $\mc{A}_i$, induces a stable equivalence of spectra
    \begin{equation*}
        \bigoplus\limits_{i=0}^q \K(\mc{A}_i) \longrightarrow \GW^{[n]}(\mc{A}).
    \end{equation*}
    \item Suppose $r$ is even. 
    Let $q = r/2$. 
    Then the functor
    \begin{equation*}
        \begin{aligned}
            \mc{A}_q \times \prod\limits_{i = 0}^{q-1} H\mc{A}_i    & \longrightarrow   \mc{A} \\
            A_q \times \prod\limits_{i = 0}^{q-1} (A_i, B_i)        & \longmapsto       A_q \oplus \left(\bigoplus\limits_{i=0}^{q-1} A_i \oplus B_i^{\vee}\right)
        \end{aligned}
    \end{equation*}
    induces a stable equivalence of spectra
    \begin{equation*}
        \GW^{[n]}(\mc{A}_q) \oplus \bigoplus\limits_{i=0}^{q-1} \K(\mc{A}_i) \longrightarrow \GW^{[n]}(\mc{A}),
    \end{equation*}
    which in turn induces a stable equivalence of the homotopy fibers of the forgetful maps $F: \GW^{[n]}(\mc{A}) \rightarrow \K(\mc{A})$ and $F': \GW^{[n]}(\mc{A}_q) \rightarrow \K(\mc{A}_q)$. 
\end{enumerate}
\end{theorem}

\begin{proof}
For the proof of (i), note that $\mc{A}_i^{\vee} \subset \mc{A}_{r-i}$.
The equivalence $\vee: \mc{A}^{\op} \ra \mc{A}$ restricts to an equivalence $\vee: \mc{A}_{i}^{\op} \rightarrow \mc{A}_{r-i}$, since the dual of a map in $\mc{A}$ lies in $\mc{A}_{r-i}$ by assumption.

Next, (ii) will be proved.
Using the notation in the statement of the theorem, let $\mc{A}_- = \langle \mc{A}_0, \dots, \mc{A}_q \rangle$ and $\mc{A}_+ = \langle \mc{A}_{q+1}, \dots, \mc{A}_r \rangle$. 
Then $\mc{A} = \langle \mc{A}_-, \mc{A}_+ \rangle$ is a semi-orthogonal decomposition of $\mc{A}$, so there is an exact sequence of pretriangulated dg categories
\begin{equation*}
    \begin{tikzcd}[cramped]
        \mc{A}_- \arrow[r]
            & \mc{A} \arrow[r]
                & \mc{A}_+.
    \end{tikzcd}
\end{equation*}
By additivity \cite[proposition 6.8]{schlichting17} for $\GW$-theory, the hyperbolic functor of the statement of (ii) induces a stable equivalence of spectra
\begin{equation*}
    \begin{tikzcd}[cramped]
        \K(\mc{A}_-) \arrow[r, "\sim"]
            & \GW^{[n]}(\mc{A})
    \end{tikzcd}
\end{equation*}
Since $\langle \mc{A}_0, \dots, \mc{A}_q \rangle$ is a semi-orthogonal decomposition of $\mc{A}_-$, additivity for connective K-theory \cite[proposition 7.10]{blumberg13} yields an equivalence
\begin{equation*}
    \begin{tikzcd}[cramped]
        \bigoplus\limits_{i = 0}^q \K(\mc{A}) \arrow[r, "\sim"]
            & \K(\mc{A}_-)
    \end{tikzcd}
\end{equation*}
and the proof of (ii) is finished by composing these two equivalences.

Finally, (iii) will be proved.
Let $\langle \mc{A}_-, \mc{A}_0, \mc{A}_+ \rangle$ be the semi-orthogonal decomposition of $\mc{A}$ with $\mc{A}_- = \langle \mc{A}_0, \dots, \mc{A}_{q-1} \rangle$ and $\mc{A}_+ = \langle \mc{A}_{q+1}, \dots, \mc{A}_{r} \rangle$.
Then $\mc{A}_-^{\vee} = \mc{A}_+$ and $\mc{A}_0^{\vee} = \mc{A}_0$. 
Hence \cite[theorem 3.5.6]{xie15} yields a stable equivalence of spectra
\begin{equation*}
    \begin{tikzcd}[cramped]
        \GW^{[n]}(\mc{A}_q) \oplus \K(\mc{A}_-) \arrow[r, "\sim"]
            & \GW^{[n]}(\mc{A}).
    \end{tikzcd}
\end{equation*}
One obtains the desired stable equivalence of spectra with another application of the additivity of connective K-theory.
It remains to show that there is a stable equivalence of homotopy fibers $\hofib(F) \simeq \hofib(F')$.
There is a commutative diagram of spectra
\begin{equation*}
    \begin{tikzcd}
        \hofib(F' \oplus \id) \arrow[r] \arrow[d, dashed, "\sim"{sloped,anchor=south}]
            & \GW^{[n]}(\mc{A}_q) \oplus \K(\mc{A}_-) \arrow[d, "\sim"{sloped,anchor=south}] \arrow[r, "F' \oplus \id"]
                & \K(\mc{A}_q) \oplus \K(\mc{A}_-) \arrow[d, "\sim"{sloped,anchor=south}] \\
        \hofib(F) \arrow[r]
            & \GW^{[n]}(\mc{A}) \arrow[r, "F"]
                & \K(\mc{A})
    \end{tikzcd}
\end{equation*}
where the vertical arrows are stable equivalences.
Thus
\begin{equation*}
    \hofib(F) \simeq \hofib(F' \oplus \id) \simeq \hofib(F') \oplus \hofib(\id),
\end{equation*}
but $\hofib(\id)$ is contractible and the result follows.
\end{proof}

\subsection{Grothendieck-Witt spectra of projective bundles}
    \label{subsection:theprojectivebundleformulaforgrothendieckwittspectra}

In this final section, the formulae for the Grothendieck-Witt spectra of general projective bundles presented in the introduction are proven.
The general strategy is to use the standard semi-orthogonal decomposition of the dg category of perfect complexes of a projective bundle, demonstrate how that decomposition behaves with respect to the duality, and to study the constituent pieces of the decomposition.

Let $X$ be a quasi-compact quasi-separated scheme over $\Spec \Z[1/2]$ with the resolution property.
As noted before, schemes with an ample family of line bundles satisfy the resolution property.
Let $\mc{E}$ be a locally free sheaf of $\OO_X$-modules of rank $r+1$ and write $\P = \P(\mc{E})$ and $\OO = \OO_{\P}$.
Also set $s = \lceil r/2 \rceil$.
Let $\pi: \P \rightarrow X$ be the structure map. 
Write $\mc{A}$ for $\Perf(\P)$. 
Let $\mc{L}$ be a line bundle on $X$ and let $L = (\OO(m) \otimes \pi^*\mc{L})[0]$, with $m \in \Z$, be the object of $\Perf(\P)$ consisting of a single copy of $\OO(m) \otimes \pi^*\mc{L}$ concentrated in cohomological degree $0$.
Let $(\Perf(\P)^{[L]},w)$ be the pretriangulated dg category with weak equivalences the quasi-isomorphisms and duality given by the mapping complex $[-,L]$; write $\natural$ for this duality on $\mc{A}$ and reserve the symbol $\vee$ for the standard duality on $\Vect(X)$, $\Vect(\P)$ and their respective categories of perfect complexes.
By corollary \ref{corollary:equivalencesquareoflinebundle} and \cite[theorem 6.5]{schlichting17}, $\OO(m)$ may be replaced by $\OO(m + 2i)$ for any $i \in \Z$ without affecting the Grothendieck-Witt theory.
Therefore, $m$ can be chosen freely up to parity in the proofs below.

The following theorem is contained in the proof of \cite[theorem 3.5.1]{schlichting11}; a version of this result in the context of stable $\infty$-categories is \cite[theorem B]{khan18}.

\begin{theorem} \label{theorem:semiorthogonaldecompositionprojectivebundle}
The following statements hold.
\begin{enumerate}[label=(\roman*)]
    \item For each $k \in \Z$, the assignment $\mc{F} \mapsto p^*\mc{F} \otimes \OO(k)$ defines a fully faithful functor $\Perf(X) \rightarrow \Perf(\P)$.
    The essential image of such a functor will be denoted $\mc{A}(-k)$. 
    \item For each $i \in \Z$, $\langle \mc{A}(i-r), \dots, \mc{A}(i) \rangle$ is a semi-orthogonal decomposition of $\Perf(\P)$. 
\end{enumerate}
\end{theorem}

\begin{proposition} \label{proposition:dualitydecompositionprojectivebundle}
With notation as in theorem \ref{theorem:semiorthogonaldecompositionprojectivebundle}, the essential image of $\mc{A}(k)$ under the duality $\natural$ is $\mc{A}(i-k)$ for all $k \in \Z$. 
\end{proposition}

\begin{proof}
First note that $\pi^*$ can be made into a dg form functor (cf. \cite[section 9.3]{schlichting17}, which discusses the functoriality of $\pi^*$).
Any object of $\mc{A}(k)$ can be written as $\pi^*M \otimes \OO(k)$ with $M \in \Perf(X)$. 
The dual of such an object satisfies
\begin{equation*}
    \begin{aligned}
        (\pi^*M \otimes \OO(k))^{\natural}
            & \cong [\pi^*M \otimes \OO(k), L] \\
            & \cong [\pi^*M, [\OO(k),L]] \\
            & \cong L \otimes [\pi^*M, \OO(-k)] \\
            & \cong \OO(m) \otimes \pi^*\mc{L}[0] \otimes \pi^*[M, \OO_X] \otimes \OO(-k) \\
            & \cong \pi^*[M, \mc{L}[0]] \otimes  \OO(m-k),
    \end{aligned}
\end{equation*}
where the isomorphisms are given by various results of section \ref{subsection:dualityonthecategoryofperfectcomplexes}, as well as standard properties of the pullback $\pi^*$.
Hence $(\pi^*M \otimes \OO(k))^{\natural}$ is an object of $\mc{A}(m-k)$.
It follows that $\mc{A}(k)^{\natural} \subset \mc{A}(m-k)$ and $\mc{A}(m-k)^{\natural} \subset \mc{A}(k)$.
As $\natural$ is an equivalence, the proof is done.
\end{proof}

The following theorem is main theorem \ref{maintheorem:projectivebundleformulawithtwistofsameparityasdimension}, the projective bundle formula for $m$ and $r$ of equal parity.
Its proof is an application of additivity for Grothendieck-Witt spectra.

\begin{theorem}[Theorem \ref{maintheorem:projectivebundleformulawithtwistofsameparityasdimension}]
\label{theorem:projectivebundleformulawithtwistofsameparityasdimension}
Recall that $s = \lceil r/2 \rceil$.
The following statements hold.
\begin{enumerate}[label=(\roman*)]
    \item If $m$ and $r$ are even, then there is a stable equivalence of spectra
    \begin{equation*}
        \begin{tikzcd}[cramped]
            \GW^{[n]}(X, \mc{L}) \oplus \K(X)^{\oplus s} \arrow[r]
                & \GW^{[n]}(\P, \pi^*\mc{L}(m)).
        \end{tikzcd}    
    \end{equation*}
    \item If $m$ and $r$ are odd, then there is a stable equivalence of spectra
    \begin{equation*}
        \begin{tikzcd}[cramped]
            \K(X)^{\oplus s} \arrow[r]
                & \GW^{[n]}(\P, \pi^*\mc{L}(m)).
        \end{tikzcd}    
    \end{equation*}
\end{enumerate}
\end{theorem}

\begin{proof}
Without loss of generality, assume that $m = -r$. 
By theorem \ref{theorem:semiorthogonaldecompositionprojectivebundle}, there is a semi-orthogonal decomposition 
\begin{equation*}
    \langle \mc{A}(-r), \mc{A}(-r+1), \dots, \mc{A}(0) \rangle
\end{equation*}
of $\mc{A}$.
The duality maps $\mc{A}(i)$ to $\mc{A}(-r-i)$ for all $i \in \Z$ by proposition \ref{proposition:dualitydecompositionprojectivebundle}.
Hence the additivity theorem \ref{theorem:generaladditivityforgw} applies and yields the desired result.
\end{proof}

This covers half of the cases of the projective bundle formula. 
Now set $m = -r-1$.
By theorem \ref{theorem:semiorthogonaldecompositionprojectivebundle}, there is a semi-orthogonal decomposition 
\begin{equation*}
    \langle \mc{A}(-r-1), \mc{A}(-r), \dots, \mc{A}(-1) \rangle
\end{equation*}
of $\mc{A}$.
By proposition \ref{proposition:dualitydecompositionprojectivebundle}, $\mc{A}(i)^{\natural} \subset \mc{A}(-r-1-i)$ for all $i \in \Z$. 
Let
\begin{equation*}
    \mc{A}_0 = \langle \mc{A}(-r), \dots, \mc{A}(-1) \rangle.
\end{equation*}
Then the constituents of $\mc{A}_0$ are exchanged by the duality, so that $\GW(\mc{A}_0)$ may be computed using the additivity theorem \ref{theorem:generaladditivityforgw}.
Furthermore, there is a quasi-exact sequence of pretriangulated dg categories with duality
\begin{equation} \label{diagram:quasiexactsequenceforperf}
    \begin{tikzcd}[cramped]
        \mc{A}_0^{[L]} \arrow[r]
            & \mc{A}^{[L]} \arrow[r]
                & (\mc{A}/\mc{A}_0)^{[L]}.
    \end{tikzcd}    
\end{equation}
Thus understanding $\GW^{[n]}((\mc{A}/\mc{A}_0)^{[L]})$ is paramount to understanding $\GW^{[n]}(\mc{A}^{[L]})$.
Fix $\Delta = \det \pi^*\mc{E}(m)$.

\begin{lemma} \label{lemma:quotientofsemiorthogonaldecompositionisbase}
There is a quasi-equivalence of pretriangulated dg categories with weak equivalences
\begin{equation*}
    F: \Perf(X) \longrightarrow \mc{A}/\mc{A}_0
\end{equation*}
given by $M \mapsto \Delta[r] \otimes \pi^*M$. 
\end{lemma}

\begin{proof}
The functor $F$ is the composition
\begin{equation*}
    \begin{tikzcd}
        \Perf(X) \arrow[r, "\pi^*"]
            & \mc{A} \arrow[r, "{\Delta[r] \otimes -}"]
                &[1em] \mc{A} \arrow[r]
                    & \mc{A}/\mc{A}_0.
    \end{tikzcd}
\end{equation*}
Denote the composition $\Perf(X) \rightarrow \mc{A}$ of the first two maps by $F'$.
Then
\begin{align*}
    F'(\det \mc{E}^{\vee}[-r])
        & = \Delta[r] \otimes \pi^*\det \mc{E}^{\vee}[-r] \\
        & \cong \det \pi^*\mc{E}[r] \otimes \pi^*\det \mc{E}^{\vee}[-r] \otimes \OO(m) \\
        & \cong \OO(m).
\end{align*}
As tensoring with $\det \mc{E}^{\vee}[-r]$ gives a self-equivalence of $\Perf(X)$, the essential image of $F'$ consists of objects of the form $\pi^*M \otimes \OO(m)$ and is therefore the subcategory $\mc{A}(-r-1)$ of $\mc{A}$ of theorem \ref{theorem:semiorthogonaldecompositionprojectivebundle}.
In particular, $F': \Perf(X) \rightarrow \mc{A}(-r-1)$ is a quasi-equivalence.
Hence $F$ factors through the canonical map $F'': \mc{A}(-r-1) \rightarrow \mc{A}/\mc{A}_0$.
Since $\langle \mc{A}(-r-1), \mc{A}_0 \rangle$ is a semi-orthogonal decomposition of $\mc{A}$, it follows that $F''$ is a quasi-equivalence.
Therefore $F$, being a composition of quasi-equivalences, is also a quasi-equivalence, as was to be shown.
\end{proof}

\begin{lemma} \label{lemma:koszulsymmetricformisquasiisomorphisminquotient}
Fix an $-r-1 \leq \ell \leq -1$ as in section \ref{subsection:constructingsymmetricformsfromkoszulcomplexes}. 
The symmetric form $\phi: M[-1] \rightarrow [M[-1], \Delta[r]]$
of \ref{proposition:symmetricformkoszulcomplexprojectivebundle} becomes a quasi-isomorphism
\begin{equation*}
    \bar{\phi}: \Delta[r] \rightarrow \OO    
\end{equation*}
in $(\mc{A}/\mc{A}_0)^{[\Delta[r]]}$, which is independent from the choice of $\ell$.
In particular, $\bar{\phi}$ defines an element of $\GW^{[r]}_0((\mc{A}/\mc{A}_0)^{[\Delta]})$.
\end{lemma}

\begin{proof}
The underlying graded $\OO$-module of the cone $C(\phi)$ of $\phi$ is given by $M \oplus [M[-1], \Delta[r]]$.
Let $-r \leq i \leq -1$.
Then $M_i = \Lambda^{-i}\pi^*\mc{E} \otimes \OO(i)$ if $i \leq \ell$ and $M_i = 0$ otherwise.
As $\OO(i) \in \mc{A}_0$, it follows that $C(\phi)_i = 0$.
Hence the image of $C(\phi)$ in $\mc{A}/\mc{A}_0$ is 
\begin{equation*}
    \begin{tikzcd}
        \overline{C(\phi)}: 
            &[-2em] K_{-r-1} \arrow[r]
                & 0 \arrow[r]
                    & \ldots \arrow[r]
                        & 0 \arrow[r]
                            & K_0,
    \end{tikzcd}
\end{equation*}
which corresponds to the image $\overline{K}$ of the Koszul complex in $\mc{A}/\mc{A}_0$, and is independent from the choice of $\ell$.   
As $K$ is acyclic and the quotient functor $\mc{A} \rightarrow \mc{A}/\mc{A}_0$ is exact, it follows that $\overline{C(\phi)}$ is acyclic as well.
Hence the image $\bar{\phi}: \Delta[r] \rightarrow \OO$ of $\phi$ in $\mc{A}/\mc{A}_0$ is a quasi-isomorphism, as was to be shown.
\end{proof}

\begin{proposition} \label{proposition:dgformfunctorbasetoquotientofsemiorthogonaldecomposition}
The quasi-equivalence $F: \Perf(X) \rightarrow \mc{A}/\mc{A}_0$ of proposition \ref{lemma:quotientofsemiorthogonaldecompositionisbase} can be made into a dg form functor
\begin{equation*}
    (F, \eta): \Perf(X)^{[0]} \rightarrow (\mc{A}/\mc{A}_0)^{[\Delta[r]]}.
\end{equation*}
In particular, there is a quasi-equivalence of pretriangulated dg categories with weak equivalences and duality
\begin{equation*}
    \Perf(X)^{[\det \mc{E}^{\vee} \otimes \mc{L}[-r]]} \simeq (\mc{A}/\mc{A}_0)^{[L]}.
\end{equation*}
\end{proposition}

\begin{proof}
By lemma \ref{lemma:koszulsymmetricformisquasiisomorphisminquotient}, $F$ can be written as the composition
\begin{equation*}
    \begin{tikzcd}
        F:
            &[-2em] \Perf(X)^{[0]} \arrow[r, "\pi^*"]
                & \mc{A}^{[0]} \arrow[r, "{(M,\phi) \otimes -}"]
                    &[2em] \mc{A}^{[\Delta[r]]} \arrow[r]
                        & (\mc{A}/\mc{A}_0)^{[\Delta[r]]},
    \end{tikzcd}
\end{equation*}
now ornamented with the dualities of each category.
Here, $(M,\phi)$ is the symmetric form of proposition \ref{proposition:symmetricformkoszulcomplexprojectivebundle}.
Note that $\pi^*$ is a dg form functor by \cite[section 9.3]{schlichting17}.
Furthermore, ${(M,\phi) \otimes -}$ is a dg form functor by proposition \ref{proposition:tensoringbysymmetricformisformfunctor}.
Finally, the quotient map 
\begin{equation*}
    \mc{A}^{[\Delta[r]]} \longrightarrow (\mc{A}/\mc{A}_0)^{[\Delta[r]]}    
\end{equation*}
is a dg form functor by construction.
Hence the quasi-equivalence of dg categories $F$ is a composition of dg form functors and therefore itself a dg form functor, as was to be shown.
Finally, twisting the duality in $\Perf(X)^{[0]}$ by the invertible complex $\det \mc{E}^{\vee} \otimes \mc{L}[-r]$ gives the desired quasi-equivalence
\begin{equation*}
    \Perf(X)^{[\det \mc{E}^{\vee} \otimes \mc{L}[-r]]} \simeq (\mc{A}/\mc{A}_0)^{[L]},
\end{equation*}
and the proof is done.
\end{proof}

Now there are no more obstacles to proving main theorem \ref{maintheorem:projectivebundleformulaforgwwithtwistofdifferentparity}, which describes the Grothendieck-Witt spectrum of projective bundles when the parities of $r$ and $m$ differ.

\begin{theorem}[Theorem \ref{maintheorem:projectivebundleformulaforgwwithtwistofdifferentparity}]
\label{theorem:projectivebundleformulaforgwwithtwistofdifferentparity}
The following statements hold.
\begin{enumerate}[label=(\roman*)]
    \item If $r$ is even and $m$ is odd, then there is a split homotopy fibration
    \begin{equation*}
        \begin{tikzcd}[column sep=small]
            \K(X)^{\oplus s-1} \arrow[r]
                & \GW^{[n]}(\P, \pi^*\mc{L}(m)) \arrow[r]
                    & \GW^{[n-r]}(X,\det \mc{E}^{\vee} \otimes \mc{L}).
        \end{tikzcd}
    \end{equation*}
    \item If $r$ is odd and $m$ is even, then there is a homotopy fibration
    \begin{equation*}
        \begin{tikzcd}[column sep=tiny]
            \GW^{[n]}(X, \mc{L}) \oplus \K(X)^{\oplus s-1} \arrow[r]
                &[-0.1em] \GW^{[n]}(\P, \pi^*\mc{L}(m)) \arrow[r]
                    &[-0.1em] \GW^{[n-r]}(X,\det \mc{E}^{\vee} \otimes \mc{L}),
        \end{tikzcd}
    \end{equation*}
    which splits if the element $\nu'$ of lemma \ref{lemma:middletermexterioralgebraissymmetric} vanishes in $\W^{[r+1]}_0(X, \det \mc{E})$.
    The condition for the splitting is satisfied e.g. if $\mc{E}$ is a trivial bundle.
\end{enumerate}
\end{theorem}

\begin{proof}
Without loss of generality, one may assume $m = -r-1$ in both cases.
The quasi-exact sequence (\ref{diagram:quasiexactsequenceforperf}) gives rise to a homotopy fibration of Grothendieck-Witt spectra 
\begin{equation*}
    \GW^{[n]}(\mc{A}_0^{[L]}) \longrightarrow \GW^{[n]}(\mc{A}^{[L]}) \longrightarrow \GW^{[n]}((\mc{A}/\mc{A}_0)^{[L]})
\end{equation*}
by \cite[theorem 6.6]{schlichting17}.
Note that $\GW^{[n]}(\mc{A}^{[L]})$ is just a different way of writing $\GW^{[n]}(\P, \pi^*\mc{L}(m))$.
As already remarked, the additivity theorem \ref{theorem:generaladditivityforgw} gives a formula for the first term $\GW^{[n]}(\mc{A}_0^{[L]})$ of both (i) and (ii). 
Thus it suffices to show that there is a quasi-equivalence
\begin{equation*}
    F: \Perf(X)^{[\det \mc{E}^{\vee} \otimes \mc{L}[-r]]} \longrightarrow (\mc{A}/\mc{A}_0)^{[L]},
\end{equation*}
but this holds by proposition \ref{proposition:dgformfunctorbasetoquotientofsemiorthogonaldecomposition}.
This yields both claimed homotopy fibrations.

If $r$ is even or the element $\nu'$ of lemma \ref{lemma:middletermexterioralgebraissymmetric} vanishes in $\W^{[r+1]}_0(X, \det \mc{E})$, then it is possible to construct a symmetric form $\psi: H \rightarrow [H, \Delta[r]]$ which is a quasi-isomorphism by proposition \ref{proposition:evenprojectivebundlecutskoszulintwo} and theorem \ref{theorem:middletermofkoszulcomplexwittnilcutskoszulintwo}.
The composition
\begin{equation*}
    \begin{tikzcd}
        F':
            &[-2em] \Perf(X)^{[0]} \arrow[r, "\pi^*"]
                & \mc{A}^{[0]} \arrow[r, "{(H,\psi) \otimes -}"]
                    &[1em] \mc{A}^{[\Delta[r]]}
    \end{tikzcd}
\end{equation*}
induces a map $\psi \cup - : \GW^{[-r]}(X, \det \mc{E}^{\vee} \otimes \mc{L}) \ra \GW^{[0]}(\P, \pi^*\mc{L}(m))$ of Grothendieck-Witt spectra.
To show that cup product with $\psi$ splits the homotopy fibration, it suffices to show that the triangle of pretriangulated dg categories with duality
\begin{equation*}
    \begin{tikzcd}[column sep=small]
        { }
            & \mc{A}^{[\Delta[r]]} \arrow[dr]
                & { } \\
        \Perf(X)^{[0]} \arrow[rr, "{F}"] \arrow[ur, "F'"]
            & { }
                & (\mc{A}/\mc{A}_0)^{[\Delta[r]]}
    \end{tikzcd}
\end{equation*}
commutes up to natural isomorphism, which amounts to showing that the image of $\psi: H \ra [H, \Delta[r]]$ in $(\mc{A}/\mc{A}_0)^{[\Delta[r]]}$ is the symmetric form $\bar{\phi}$ of lemma \ref{lemma:koszulsymmetricformisquasiisomorphisminquotient}. 
If $r$ is even, this follows from proposition \ref{proposition:evenprojectivebundlecutskoszulintwo}.
If $r$ is odd, then the construction of $(H,\psi)$ of theorem \ref{theorem:middletermofkoszulcomplexwittnilcutskoszulintwo} shows that $\mc{N}$, $\mc{P}$ and $\mc{S}$ are all contained in $\mc{A}_0$.
It follows that the image of $\psi$ in $(\mc{A}/\mc{A}_0)^{[\Delta[r]]}$ is isomorphic to $\bar{\phi}$, which completes the proof.

Note that lemma \ref{lemma:quotientbundleoddrankimplieswittnil} provides a sufficient condition for $\nu$ to vanish, and that this condition holds in particular if $\mc{E}$ is a trivial bundle.
\end{proof}